\documentclass[a4paper,10pt]{article}
\usepackage[latin1]{inputenc}
\usepackage{amsthm}
\usepackage{amsmath,amssymb,amsfonts}
\usepackage{hyperref}
\usepackage[english]{babel}
\usepackage[mathscr]{euscript}
\usepackage{enumerate}
\usepackage{xspace}
\usepackage[mathlines]{lineno}
\makeatletter
\newdimen\mymathindent
\newenvironment{bulletequation}%
    {\@beginparpenalty\predisplaypenalty
     \@endparpenalty\postdisplaypenalty
     \refstepcounter{equation}%
     \trivlist \item[]\leavevmode
       \hb@xt@\linewidth\bgroup $\m@th
         \displaystyle
         \hskip\mymathindent}%
        {$\hfil 
         \displaywidth\linewidth\hbox{\@eqnnum}%
       \egroup
     \endtrivlist}
\makeatother

\usepackage{url}
\numberwithin{equation}{section}
\newtheorem{thm}{Theorem}[section]

\newtheorem{prop}[thm]{Proposition}
\newtheorem{lm}[thm]{Lemma}
\newtheorem{cor}[thm]{Corollary}

\theoremstyle{definition}

\newtheorem{Rq}[thm]{Remark}

\theoremstyle{remark}

\theoremstyle{plain}

\DeclareMathAlphabet{\calptmx}{OMS}{ztmcm}{m}{n}

\newcommand{\ds}{{\displaystyle}}
\renewcommand{\d}{d}

\newcommand{\C}{\mathbb{C}}
\newcommand{\R}{\mathbb{R}}

\newcommand{\Cc}{\mathcal{C}}

\newcommand{\Sb}{\mathbb{S}}
\newcommand{\non}{\noindent}

\newcommand{\ve}{\vspace{3mm}}
\author{Bakri Laurent}
\title{{\bf\LARGE Vanishing order of solutions to Schr\"odinger equation}}
\begin{document}
\selectlanguage{english}
\date{}
\maketitle
 \begin{center}\small{E-mail : laurent.bakri@gmail.com}\end{center}
\begin{abstract}
 On a compact manifold, we give a Carleman estimate on the operator $\Delta +W$ when $\Delta$ is the laplacian operator  and $W$ is a bounded function.
  We then deduce quantitative uniqueness result for solutions to $$\Delta u +Wu=0$$ using doubling estimates. 
In particular this show that the vanishing order is everywhere less than $C_1\|W\|_\infty^{\frac{2}{3}}+C_3$.
 Finally we investigate the sharpness of this results by constructing a potential $W$ with compact support on $\Sb^2$.

\end{abstract}

\section{Introduction and statement of the results}
Let $(M,g)$ be a  compact, connected, $n$-dimensional smooth Riemannian manifold, $\Delta$ the Laplace operator on $M$ and $W$
 a bounded function on $M$.
 If $u$ is a non trivial solution to \begin{equation}\label{E}\Delta u+Wu=0,\end{equation} 

\non we consider  the possible vanishing order, depending on $W$,
 of $u$ in any point.
In the case that $W$ is a constant, \emph{i.e}, when dealing with the eigenfunctions of the Laplacian, 
it is a well known result of H. Donnelly and C. Fefferman \cite{DF1} that the vanishing order is erverywhere bound by $c\sqrt{\lambda}$, 
with $c$  constant depending only on $M$. 
When $W$ is a $\Cc^1$ function, the author have shown in \cite{B1} that the vanishing order 
is bounded by $C_1\sqrt{\|W\|_{\Cc^1}}+C_2$, where $\|W\|_{\Cc^1}=\sup_M|W|+\sup_M |\nabla W|$ and the norm
  $|\nabla W|$ is taking with respect to the metric $g$. 
In this paper we investigate the possible vanishing order 
of solutions to \eqref{E} when $W$ is only a bounded function. For $W$ a real valued function,
I. Kukavica established, in \cite{K}, the following uniform upper bound :
 $$C(\mathrm{sup}(W_-)^{\frac{1}{2}}+(\mathrm{osc}({W}))^2+1)\\$$

 where $W_-=\mathrm{min}(W,0)$ 
and $\mathrm{osc}(W)=\sup W - \inf W$. In particular the vanishing order of solutions to \eqref{E} is uniformly bounded 
by $$C(1+\|W\|_\infty^2).$$

\non Our first goal is to sharpen this last estimate. We shall proove the following: 

 \ve\begin{thm}\label{t1}  There exist  two non-negatives constants $C_1,C_2$ depending only on $M$, such that, for any solutions $u$ 
to \eqref{E}
 and for any point $x_0$ in $M$,
the vanishing order of any non-zero solutions to \eqref{E} is everywhere bounded by $$C_1\|W\|^{\frac{2}{3}}_{\infty}+C_2.$$
\end{thm} 
Note here that the potential $W$ and the solution $u$ may take their values in $\C$.
This result is sharp in the following sense, considering complex valued solutions of \eqref{E}, the exponant  $2/3$ on $\|W\|_\infty$ 
is the lowest one can obtain in the upper bound on the vanishing order of solutions.  
More precisely we will show the following : 
\begin{thm}
\label{EX}
There exists a constant $C$ such that, if $N>0$ is an arbitrary great number, there exists  a function $W \in L^\infty(\mathbb{S}^2,\C)$ 
with $N\geq C \|W\|_\infty^{\frac{2}{3}}$ and $u\in \mathcal{C}^2(\mathbb{S}^2,\C)$ solution to \eqref{E} 
which vanishes with order $N$  in $P$. Moreover $W$ can be choosen of compact support with  
 $$\mathrm{supp}(W)\subset\left(\Sb^2\setminus(\{P\}\cup \{Q\})\right)$$ where 
the points $P,Q$ are antipodals.  

\end{thm}
\non In the case one only consider real valued functions, the upper bound  $$C_1\|W\|^{\frac{1}{2}}_\infty+C_2$$ 
is expected (see by example \cite{K}).
However this seems to be a difficult problem which, to the knowledge of the author, has not be solved yet. 
As mentionned in \cite{BK}, since Carleman estimates don't distinguish between the real and complex case, 
it seems difficult to obtain such result with this method.
 

The method of \cite{K} was based on the frequency function \cite{GL,L} when our's relies on Carleman estimates
 \cite[...]{DF1,JK,Kenig,KT2,KT1,Regs,Sogge}.
 This two methods are the principal way to obtain quatitative uniqueness results for solutions of partial differential equations.
The first section is devoted to study the vanishing order of solutions to \eqref{E}, by using Carleman inequalities. 
This is inspired by the works of 
H. Donnelly and C. Fefferman for eigenfunctions of the Laplacian.
In particular in section 2 we first established a Carleman estimate which is only true for great enough paramater
 $\tau\geq\tau_0$ and we state explicity how $\tau_0$ depends on the potential $W$. 
We proove this estimate by elementary method and for a certain family of weight functions. Then we used the
 Carleman estimate with a special choice of weight functions to obtain a Hadamard's type three circles theorem 
and doubling inequalities on solution of \eqref{E}. This will proove theorem \ref{t1}. 
Section 3 is devoted to proove theorem \ref{EX}.  We construct a sequence $(u_k,W_k)$ verifying 
 $\Delta u_k+W_ku_k=0$ on the two dimensionnal sphere which shows
that the sharp upper bound $C_1\|W\|_\infty^{\frac{2}{3}}+C_2$ on the vanishing order can be obtained with $W_k$ 
of support in the neighborhood of two antipodals points.\\ 
 
It must be emphasis that both the upper bound and his sharpness are already known. 
In \cite{Kenig}, C. A. Kenig  established this results, with similar works on $\R^n$ which clearly imply 
the present results of this paper.

\vspace{3mm}
 {\bf Notations} \\
For a fixed point $x_0$  in $M$ we will use the following standard notations:
\begin{itemize}
\item $r:=r(x)=d(x,x_0)$ the Riemannian distance from $x_0$,
\item $v_g$ the volume form induced by $g$,
\item $B_r:=B_r(x_0)$ the geodesic ball centered at  $x_0$ of radius $r$,
\item $A_{r_1,r_2}:=B_{r_2}\setminus B_{r_1}$,
\item $\|\cdot\|$ the  $L^2$ norm on $M$ and $\|\cdot\|_A$ the $L^2$ norm on the (measurable) set $A$.
 In case $T$ is a vector field (or a tensor), 
it has to be understood as $\||T|_g\|$. 
\item $c$, $C$, $c_i$ and $C_i$ for $i=1,2,\cdots$ are constants wich may depend on $(M,g)$ and other 
quantities such that the weight functions in Carleman estimates (section 2.1), but not 
on the potential $W$ or solutions $u$. Their values might change from one line to another. 
\end{itemize}
\section{Vanishing order} Recall that Carleman estimates are weighted integral inequalities with a weight 
function $e^{\tau\phi}$, where the function 
$\phi$ satisfy some convexity properties, see by example \cite{Horm2,I,JL}. 
We first proove in section 2.1 an $L^2$, singular weighted, Carleman inequality on the operator $\Delta +W$, 
for some class of weight functions satisfying convenient properties (see \eqref{f} and \eqref{weight} below). 
After that, in sections 2.2 and 2.3, we  make a particular choice of weight function which will allows us 
to derive a doubling estimate on solutions to \eqref{E} .
\subsection{Carleman estimate}

Let us first define the class of (singular) weight functions we will work with.\newline  
Let  $f\: : ]-\infty,T[ \rightarrow \mathbb{R}$ of  class $\mathcal{C}^3$, and assume  
there are constants  $\mu_i>0$,  $i=1,\cdots,4$,  such that  :
\begin{equation}\label{f}
\begin{array}{lcr}
&0< \mu_1\leq f^\prime(t)\leq \mu_2&\\
&\displaystyle{\mu_3|f^{(3)}(t)|\leq -f^{\prime\prime}(t)\leq \mu_4}&,\ \forall t\in]-\infty,T[\\
&\displaystyle{\lim _{t\rightarrow-\infty}-e^{-t}f^{\prime\prime}(t)}=+\infty &
\end{array}
\end{equation}
{\bf Example :} It is clear that the functions defined by $f_\varepsilon(t)=t-e^{\varepsilon t},$ 
satisfy the conditions \eqref{f}
 provided $0<\varepsilon<1$ and $T$ is a large negative number.\vspace{0,2cm} \\
Finally we define our weight function as \begin{equation}\label{weight}\phi(x)=-f(\ln r(x)).\end{equation}

Now we can state the main result of this section:

\begin{thm}\label{ticsl}
There exist positive constants $R_0, C,C_1,C_2$, which depend only on  $M$ and $f$, such that, 
for any  $x_0\in M$, any $\delta\in(0,R_0)$, any $W\in L^\infty(M),\ $any 
 $u\in \Cc^\infty_0\left(B_{R_0}(x_0)\setminus B_{\delta}(x_0)\right)$ and any
 $\tau \geq C_1\|W\|^{\frac{2}{3}}_\infty+C_2$, one has

\begin{equation}\label{Si}\begin{split}
C\left\|r^2e^{\tau\phi}(\Delta +W)u r^{-n/2}\right\|^2
&
\geq
\tau^3\left\|\sqrt{|f^{\prime\prime}(\mathrm{ln}\:r)|}e^{\tau\phi}ur^{-n/2}\right\|^2 
\\
+\ \tau^2\delta\left\| r^{-\frac{1}{2}}e^{\tau\phi}ur^{-n/2}\right\|^2
&
+
\tau\left\| r\sqrt{|f^{\prime\prime}(\mathrm{ln}\:r)|}e^{\tau\phi}|\nabla u|^2r^{-n/2}\right\|^2.
\end{split}\end{equation}
\end{thm}
\noindent The following lemma contains the crucial part of theorem \ref{ticsl}:
\begin{lm}\label{lm}
There exist positive constants $R_0, C,C_1$, which depend only on  $M$ and $f$, such that, 
for any  $x_0\in M$, \ any  $u\in C^\infty_0\left(B_{R_0}(x_0)\setminus \{x_0\})\right)$ and any
 $\tau \geq C_1$, one has    
\begin{equation}\label{SS}
\begin{split}
C\left\| r^2e^{\tau\phi}|\Delta u|r^{-n/2}\right\|^2 
\geq
&  {}\ 
\tau^3\left\|\sqrt{|f^{\prime\prime}(\mathrm{ln}r)|}e^{\tau\phi}ur^{-n/2}\right\|^2
 \\
 &
+
  \tau\left\| r\sqrt{|f^{\prime\prime}(\mathrm{ln}r)|}e^{\tau\phi}|\nabla u|r^{-n/2}\right\|^2.
\end{split}
\end{equation}
 
\end{lm}

\begin{Rq}
The exponent $3$ on $\tau$ in lemma \ref{lm} will play a important role in the following (compare to \cite{DF1}).
 Like in \cite{DF1} the important statement in theorem \ref{ticsl} is the one involving the parameter 
$\tau$ : $$\tau\geq C_1\|W\|^{\frac{2}{3}}+C_2.$$
\end{Rq}

\begin{proof}[proof of lemma \ref{lm}]
Without loss of generality, we  may suppose that all functions are real.   
We now introduce the polar geodesic coordinates $(r,\theta)$ near $x_0$. Using Einstein notation, the Laplace operator takes the form : 
$$r^2\Delta u=r^2\partial_r^2u+r^2\left(\partial_r\ln(\sqrt{\gamma})+\frac{n-1}{r}\right)\partial_ru+
\frac{1}{\sqrt{\gamma}}\partial_i(\sqrt{\gamma}\gamma^{ij}\partial_ju),$$
where $\displaystyle{\partial_i=\frac{\partial}{\partial\theta_i}}$ and for each fixed $r$,  $\ \gamma_{ij}(r,\theta)$  
is a metric on \: $\Sb^{n-1}$ and $\displaystyle{\gamma=\mathrm{det}(\gamma_{ij})}$.\newline
Since $(M,g)$ is smooth, we have for $r$ small enough :  
\begin{eqnarray}\label{m1}
\partial_r(\gamma^{ij})&\leq& C (\gamma^{ij})\ \ \ \mbox{(in the sense of tensors)}; 
\nonumber \\
|\partial_r(\gamma)|&\leq& C;\\
C^{-1}\leq\gamma&\leq& C.\nonumber 
 \end{eqnarray}
\non Set  $r=e^t$, 
the function $u$ is then supported in $]-\infty,T_0[\times\mathbb{S}^{n-1},$ 
where $|T_0|$ will be chosen large enough. In this new variables, we can write : 
$$e^{2t}\Delta u=\partial_t^2u+(n-2+\partial_t\mathrm{ln}\sqrt{\gamma})\partial_tu
+\frac{1}{\sqrt{\gamma}}\partial_i(\sqrt{\gamma}\gamma^{ij}\partial_ju).$$
The conditions (\ref{m1}) become
\begin{eqnarray}\label{m2}
\partial_t(\gamma^{ij})&\leq& Ce^t (\gamma^{ij})\nonumber\ \ \ \mbox{(in the sense of tensors)};  \\
|\partial_t(\gamma)|&\leq& Ce^t;\\
C^{-1}\leq\gamma &\leq & C. \nonumber
\end{eqnarray}
It will be useful to note that this properties imply 
\begin{equation}\label{ln}|\partial_t\ln\sqrt{\gamma}|\leq  Ce^t
                           \end{equation}
Now we introduce the conjugate operator : 
\begin{equation}\begin{array}{rcl}
L_\tau(u)&=&e^{2t}e^{\tau\phi}\Delta(e^{-\tau\phi}u)\\
               &=&\partial^2_tu+\left(2\tau f^\prime+n-2+\partial_t\mathrm{ln}\sqrt{\gamma}\right)\partial_tu\\
               &+&\left(\tau^2f^{\prime^2}+(n-2)\tau f^{\prime}+\tau f^{\prime\prime}+\tau\partial_t\mathrm{ln}\sqrt{\gamma}f^{\prime}\right)u\\
               &+&\Delta_\theta u,
     \end{array}\end{equation}
with $$\Delta_\theta u=\frac{1}{\sqrt{\gamma}}\partial_i\left(\sqrt{\gamma}\gamma^{ij}\partial_ju\right).$$

\noindent It will be useful for us to introduce the following $L^2$ norm on $]-\infty,T_0[\times\Sb^{n-1} $: 
$$\|V\|_f^2=\int_{]-\infty,T_0[\times\Sb^{n-1}} |V|^2\sqrt{\gamma}{f^{\prime}}^{-3}dtd\theta,$$ where $d\theta$
 is the usual measure on $\Sb^{n-1}$.
The corresponding inner product is denoted by  $\left\langle\cdot,\cdot\right\rangle_f$ , 
\emph{i.e} $$\langle u,v\rangle_f = \int uv\sqrt{\gamma}{f^{\prime}}^{-
3}dtd\theta.$$ 
\noindent We will estimate from below $\|L_\tau u\|^2_f$ by using elementary algebra and integrations by parts. 
We are concerned, in the computation, by the power of $\tau$ 
and  exponenial decay when $t$ goes to $-\infty$. First by triangular inequality one has
\begin{equation}
\|L_\tau(u)\|_f^2\geq \frac{1}{2}I-I\!I ,
\end{equation}
with 
\begin{equation}\begin{array}{rcl}
I&=& \left\|\tau^2f^{\prime^2}u+(n-2)\tau f^\prime u+\partial^2_tu
+2\tau f^\prime\partial_tu+\Delta_\theta u\right\|^2_f,\vspace{0,2cm}\\
I\!I&=&\left\|\tau f^{\prime\prime}u+ \partial_t\mathrm{ln}\sqrt{\gamma}f^{\prime}u+(n-2)\partial_tu+\partial_t\ln\sqrt{\gamma}\partial_tu\right\|^2_f.
\end{array}\end{equation}
We will now derive a lower bound of $I$. ( We will able to absorb $I\!I$ in $I$ later.) 
To compute $I$ we write :

$$I=I_1+I_2+I_3,$$ with 

\begin{equation}\begin{array}{rcl}
I_1&=&\left\|\partial^2_tu+(\tau^2f^{\prime^2}+(n-2)\tau f^\prime ) u+\Delta_\theta u\right\|_f^2\vspace{0,2cm}\\
I_2&=&\left\|2\tau f^\prime\partial_tu\right\|_f^2\vspace{0,2cm}\\
I_3&=&2\left<2\tau f^{\prime}\partial_tu\ ,\ 
\partial^2_tu+(\tau^2f^{\prime^2}+(n-2)\tau f^\prime )u+\Delta_\theta u\right>_f\label{I}
\end{array}\end{equation}
           
\non In order to make explicit the estimation of $I_3$ we write it in a convenient way: 
\begin{equation}
 I_3=J_1+J_2+J_3
\end{equation}
where, using $2\partial_tu\partial_t^2u=\partial_t(|\partial_tu|^2)$, the integrals $J_i$ are defined by :
\begin{eqnarray}\label{I3}
J_1&=& \int (2\tau f^\prime) \partial_t(|\partial_tu|^2)f^{\prime^{-3}}\sqrt{\gamma}\d t\d\theta,\\
J_2&=&4\int \tau f^\prime\partial_tu\partial_i\left(\sqrt{\gamma}\gamma^{ij}\partial_ju\right)f^{\prime^{-3}}\d t\d\theta,\\
J_3&=&\int\left(2\tau^3+2(n-2)\tau^2f^{\prime^{-1}}\right) \partial_t(|u|^2)\sqrt{\gamma}\d t\d\theta,
\end{eqnarray}
\noindent Now we will use integration by parts to estimate $J_i$. 
We recall that $f$ is radial. We find that  :
\begin{equation*}\begin{array}{rcl}
J_1&=& 4\tau \int f^{\prime\prime}|\partial_tu|^2f^{\prime^{-3}}\sqrt{\gamma}dtd\theta\\
&-&\int2\tau f^\prime\partial_t\mathrm{ln}\sqrt{\gamma}|\partial_{t}u|^2f^{\prime^{-3}}\sqrt{\gamma}dtd\theta.\end{array}
\end{equation*}
Recall that \eqref{m2} implies that $|\partial_t\ln\sqrt{\gamma}|\leq Ce^t$. Then properties \eqref{f} on $f$ gives, for large $|T_0|$ 
that $|\partial_t\ln\sqrt{\gamma}|$ is small compared to $|f^{\prime\prime}|$. Therefore one has   
\begin{equation}\label{J_1}J_1\geq -5\tau\int |f^{\prime\prime}|\cdot|\partial_tu|^2f^{\prime^{-3}}\sqrt{\gamma}dtd\theta.\end{equation}

\noindent 
In order to estimate $J_2$ we first integrate by parts with respect to $\partial_i$ : 
\begin{equation*}\begin{array}{rcl}
J_2
   &=&-2\int2\tau f^{\prime}\partial_t\partial_iu\gamma^{ij}\partial_juf^{\prime^{-3}}\sqrt{\gamma}dtd\theta.                  
     \end{array}
\end{equation*}            
Then we integrate by parts with respect to $\partial_t$. We get : 
\begin{equation*}\begin{array}{rcl}
J_2&=&-4\tau\int f^{\prime\prime}\gamma^{ij}\partial_iu\partial_juf^{\prime^{-3}}\sqrt{\gamma}dtd\theta\\
&+&\int2\tau f^{\prime}\partial_t\mathrm{ln}\sqrt{\gamma}\gamma^{ij}\partial_iu\partial_juf^{\prime^{-3}}\sqrt{\gamma}dtd\theta\\
&+&\int2\tau f^\prime\partial_t(\gamma^{ij})\partial_iu\partial_juf^{\prime^{-3}}\sqrt{\gamma}dtd\theta.
\end{array}
\end{equation*}
We denote  $|D_\theta u|^2=\partial_iu\gamma^{ij}\partial_ju$. Now using that $-f^{\prime\prime}$ is non-negative and $\tau$ is large, 
the conditions \eqref{f}  and \eqref{m2} gives for $|T_0|$ large enough:
\begin{equation}\label{J_2}
 J_2\geq 3\tau\int|f^{\prime\prime}|\cdot|D_\theta u|^2f^{\prime^{-3}}\sqrt{\gamma}dtd\theta.
\end{equation}

Similarly computation of $J_3$ gives :
\begin{equation}\label{J_31}\begin{split}
J_3&=-\int (2\tau^3+2(n-2)\tau^2f^{\prime^{-1}})\partial_t\mathrm{ln}\sqrt{\gamma}|u|^2\sqrt{\gamma}dtd\theta\\
& {} +2\int(n-2)\tau^2f^{\prime\prime}f^{\prime}|u|^2f^{\prime^{-3}}\sqrt{\gamma}dtd\theta
\end{split}\end{equation}
 From  \eqref{f} , \eqref{m2} we have provided that $\tau$ and $|T_0|$ are large enough :\newline
\begin{equation}\label{J_3}\begin{split}
J_3&\geq-3\tau^3\int e^t|u|^2f^{\prime^{-3}}\sqrt{\gamma}dtd\theta\\
& -c\tau^2\int |f^{\prime\prime}|\cdot|u|^2f^{-3}\sqrt{\gamma}\d t d\theta
\end{split}\end{equation}
%

 \noindent Thus far, using \eqref{J_1},\eqref{J_2} and \eqref{J_3}, we have : 
 \begin{equation}\label{I_3}\begin{array}{rcl}
 I_3 &\geq & 4\tau \int\left|f^{\prime\prime}\right|\left|D_\theta u\right|^2f^{\prime^{-3}}\sqrt{\gamma}dtd\theta -c\tau^3\int e^t|u|^2f^{\prime^{-3}}\sqrt{\gamma}dtd\theta\\
 &-&c\tau\int\left|f^{\prime\prime}\right|\left|\partial_t u\right|^2f^{\prime^{-3}}\sqrt{\gamma}dtd\theta -c\tau^2\int |f^{\prime\prime}|\cdot|u|^2f^{-3}\sqrt{\gamma}\d t d\theta.
 \end{array}\end{equation}

\noindent Now we consider  $I_1$ :
$$I_1=\left\|\partial^2_tu+(\tau^2
f^{\prime^2}+(n-2)\tau {f^\prime})u+\Delta_\theta u\right\|_f^2 .$$
Let $\rho>0$ a small number to be chosen later. Since  $|f^{\prime\prime}|\leq\mu_4$ and $\tau\geq1$, we have : \newline
\begin{equation}I_1\geq\frac{\rho}{\tau}I_1^\prime,\end{equation}
where $I_1^\prime$ is defined by :
\begin{equation}I_1^\prime=\left\|\sqrt{|f^{\prime\prime}|}\left[\partial^2_tu+(\tau^2f^{\prime^2}+(n-2)\tau f^\prime)u+\Delta_\theta u\right]\right\|_f^2 \end{equation}
and one has \begin{equation}I_1^\prime=K_1+K_2+K_3,\end{equation} 
with 
\begin{equation}\label{Ki}
\begin{array}{rcl}
  K_1&=&\left\|\sqrt{|f^{\prime\prime}|}\left(\partial_t^2u+\Delta_\theta u\right)\right\|_f^2, \\
K_2&=&\left\|\sqrt{|f^{\prime\prime}|}(\tau^2
f^{\prime^2}+(n-2)\tau {f^\prime})u\right\|_f^2,  \\
K_3&=&2\left\langle\left(\partial_t^2u+\Delta_\theta u\right)\left|f^{\prime\prime}\right|,(\tau^2
f^{\prime^2}+(n-2)\tau {f^\prime})u\right\rangle_f.
\end{array}\end{equation}
We first estimate $K_3$. We set 
$$A:=A(t)=(\tau^2
f^{\prime^2}f^{\prime\prime}+(n-2)\tau {f^\prime}f^{\prime\prime}).$$ 
Using properties \eqref{f}, we notice the following estimates 
\begin{align}\label{A1}
 |A|\leq {} & C\tau ^2|f^{\prime\prime}| \\
\label{A2}|\partial_t A|\leq {} & C\tau^2 |f^{\prime\prime}|
\end{align}
\noindent 
Now, recall that $|f^{\prime\prime}|=-f^{\prime\prime}$, we can write 
$$K_3=-2\int A(\partial_t^2u+\Delta_\theta u)uf^{\prime^{-3}}\sqrt{\gamma}\d t\d \theta$$

Integrating by parts gives : 
\begin{equation}\label{K_3}\begin{array}{rcl}
K_3&=&2\int A|\partial_tu|^2f^{\prime^{-3}}\sqrt{\gamma}dtd\theta\\
&+&2\int \partial_tA\partial_tuuf^{\prime^{-3}}\sqrt{\gamma}dtd\theta \\
&-&6\int A(t)f^{\prime\prime}f^{\prime^{-1}}\partial_tuuf^{\prime^{-3}}\sqrt{\gamma}dtd\theta \\
&+&2\int A(t)\partial_t\mathrm{ln}\sqrt{\gamma}\partial_tuuf^{\prime^{-3}}\sqrt{\gamma}dtd\theta\\
&+&2\int A(t)|D_\theta u|^2f^{\prime^{-3}}\sqrt{\gamma}dtd\theta\\
\end{array}\end{equation}
 
Now since  $2\partial_tuu\leq u^2+|\partial_tu|^2$, we can  use properties on $f$ \eqref{f}, $\gamma$ \eqref{m2} and $A$ (\ref{A1},\ref{A2}) to get

\begin{equation}
 K_3\geq-c\tau^2\int|f^{\prime\prime}|\left(|\partial_tu|^2+|D_\theta u|^2+|u|^2\right)f^{\prime^{-3}}\sqrt{\gamma}dtd\theta\\
\end{equation}
To estimate $K_2$ \eqref{Ki}, we notice that

\begin{align}
K_2\geq  & {} c_1\tau^4\|\sqrt{|f^{\prime\prime}|}|u|\|^2_f 
- c_2\tau^2\|\sqrt{|f^{\prime\prime}}|u\|^2_f\nonumber\\ 
K_2 \geq & {} C\tau^4\|\sqrt{|f^{\prime\prime}|}u\|^2_f
\end{align}
and since  $K_1\geq 0$ ,  
\begin{equation}\label{I_1}\begin{array}{rcl}
I_1&\geq&-\rho c\tau\int|f^{\prime\prime}|\left(|\partial_t u|^2+|D_\theta u|^2\right)f^{\prime^{-3}}\sqrt{\gamma}dtd\theta\\
&+&C\tau^{3}\rho\int|f^{\prime\prime}||u|^2f^{\prime^{-3}}\sqrt{\gamma}dtd\theta.
\end{array}\end{equation}

\noindent Then using  \eqref{I_3} and \eqref{I_1} 
\begin{multline}\label{pro}
I^2  \geq  4\tau^2\| f^\prime\partial_tu\|_f^2+4\tau\int|f^{\prime\prime}||D_\theta u|^2f^{\prime^{-3}}\sqrt{\gamma}dtd\theta \\
 + C\tau^{3}\rho\int|f^{\prime\prime}||u|^2f^{\prime^{-3}}\sqrt{\gamma}dtd\theta-c\tau^3\int e^t|u|^2f^{\prime^{-3}}\sqrt{\gamma}dtd\theta\\
-\rho c\tau\int|f^{\prime\prime}|\left(|u|^2+|\partial_tu|^2+|D_\theta u|^2\right)f^{\prime^{-3}}\sqrt{\gamma}dtd\theta\\
-c\tau^2\int |f^{\prime\prime}|\cdot|u|^2f^{-3}\sqrt{\gamma}\d t d\theta.\phantom{aaaaaaaaaaaaaaaa}
 \end{multline}

\noindent
Now one needs to check that every non-positive terms in the right hand side of \eqref{pro} can be absorbed in the first three terms. 
\\ First fix $\rho$ small enough such that $\displaystyle{\rho\leq\frac{2}{c}}$, where $c$ is the constant appearing in \eqref{pro},
 then we obviously have
$$\rho c\tau\int|f^{\prime\prime}|\cdot|D_\theta u|^2{f^{\prime}}^{-3}\sqrt{\gamma}dtd\theta\leq
 2\tau\int|f^{\prime\prime}|\cdot|D_\theta u|^2f^{\prime^{-3}}\sqrt{\gamma}dtd\theta.$$
 But the following integral  
$$ -\rho c\tau\int|f^{\prime\prime}|\left(|u|^2+|\partial_tu|^2\right)f^{\prime^{-3}}\sqrt{\gamma}dtd\theta,$$
 can also  be absorbed by comparing powers of $\tau$ of the positive terms. 
Finally since conditions \eqref{f} imply that $e^t$ is small compared to $|f^{\prime\prime}|$, 
we can absorb   $-c\tau^3e^t|u|^2$ in $C\tau^{3}\rho|f^{\prime\prime}||u|^2$.

\noindent 
\begin{equation}\label{ssu}\begin{array}{rcl}
 I &\geq &C\tau^2\int|\partial_t u|^2f^{\prime^{-3}}\sqrt{\gamma}dtd\theta+C\tau\int|f^{\prime\prime}||D_\theta u|^2f^{\prime^{-3}}\sqrt{\gamma}dtd\theta\\
&+ &C\tau^{3}\int|f^{\prime\prime}||u|^2f^{\prime^{-3}}\sqrt{\gamma}dtd\theta
\end{array}\end{equation}

\noindent We can now check that $I\!I$ can be absorbed in $I$. Since for  $|T_0|$ and $\tau$ large enough, we have : 
$$I\!I \leq c(\tau^2\|f^{\prime\prime}u\|^2_f+\|e^tu\|_f^2+\|\partial_tu\|^2_f+\|e^t\partial_tu\|_f^2),$$ 
All this terms can easily by absorbed in \eqref{ssu}.\\
Then we obtain \begin{equation}\|L_\tau u\|_f^2\geq C\tau^3\|\sqrt{|f^{\prime\prime}|} u\|_f^2+C\tau^2\|\partial_t u\|_f^2
+C\tau\|\sqrt{|f^{\prime\prime}|}\:|D_\theta u|\|_f^2 .\end{equation}
The condition   $\mu_4\geq |f^{\prime\prime}|$ force
 $\|\partial_tu\|^2_f\geq \mu_4^{-1}\|\sqrt{|f^{\prime\prime}|}\partial_tu\|^2_f$, then replacing $\tau^2$ by $\tau$ in the term involving $\partial_tu$ we can state that
\begin{equation}\|L_\tau u\|_f^2\geq C\tau^3\|\sqrt{|f^{\prime\prime}|} u\|_f^2+C\tau\|\sqrt{|f^{\prime\prime}|}\partial_t u\|_f^2
+C\tau\|\sqrt{|f^{\prime\prime}|}\:|D_\theta u|\|_f^2 .\end{equation}
Then  we set  $u=e^{\tau\phi}v$. We first note that 
$$\tau\|\sqrt{|f^{\prime\prime}|}\partial_t(e^{\tau\phi}v)\||_f^2\geq\frac{\tau}{2}\|\sqrt{|f^{\prime\prime}|}e^{\tau\phi}\partial_tv\|_f^2-\mu_2\tau^3\|\sqrt{|f^{\prime\prime}|}e^{\tau\phi}v\|_f^2 $$
 We can deduce from the definition of $\phi$, for $C$ large enough that 
\begin{equation}\label{ssu2}\begin{array}{rcl}\|e^{2t}e^{\tau\phi}(\Delta v)\|_f^2&\geq& C\tau^3\|\sqrt{|f^{\prime\prime}|}e^{\tau\phi} v\|_f^2\\
+
C\tau\|e^{\tau\phi}\partial_t v\|_f^2&+&C\tau\|\sqrt{|f^{\prime\prime}|}e^{\tau\phi} D_\theta v\|_f^2\end{array}.\end{equation}
It remains  to compare $\|\cdot\|_f$ to the usual $L^2$ norm. First note that since $0<\mu_1\leq f^\prime\leq\mu_2$,
\eqref{ssu2} we can drop the term $(f^\prime)^{-3}$ in the integrals of \eqref{ssu2} . Recall that in polar coordinates $(r,\theta)$ the volume element 
is $r^{n-1}\sqrt{\gamma}drd\theta$ and that $r=e^t$, we can deduce from \eqref{ssu2} that : 
\begin{equation}
 \begin{array}{rcl}
  \|r^2e^{\tau\phi}\Delta vr^{-\frac{n}{2}}\|^2&\geq&C\tau^3\|\sqrt{|f^{\prime\prime}|}e^{\tau\phi}vr^{-\frac{n}{2}}\|^2\\
&+&C\tau\|r\sqrt{|f^{\prime\prime}|}e^{\tau\phi}|\nabla v|r^{-\frac{n}{2}}\|^2.
 \end{array}
\end{equation}

\non This achieves the  proof of lemma \ref{lm}.\end{proof}

\begin{proof}[proof of theorem \ref{ticsl}] 

Now  we additionaly suppose that  $\mathrm{supp}(u)\subset\{x\in M; r(x)\geq\delta>0\}$ and define $T_1=\ln\delta$.\newline

\noindent Cauchy-Schwarz inequality apply to   $$\int\partial_t(u^2)e^{-t}\sqrt{\gamma}dtd\theta=2\int u\partial_tue^{-t}\sqrt{\gamma}dtd\theta$$ 
gives
\begin{equation}\label{d1}
 \int\partial_t(u^2)e^{-t}\sqrt{\gamma}dtd\theta\leq 2\left(\int\left(\partial_tu\right)^2e^{-t}\sqrt{\gamma}dtd\theta \right)^{\frac{1}{2}}
\left(\int u^2e^{-t}\sqrt{\gamma}dtd\theta\right)^{\frac{1}{2}}.
\end{equation}
On the other hand, integrating by parts gives
 \begin{equation}
      \int\partial_t(u^2)e^{-t}\sqrt{\gamma}dtd\theta = \int u^2e^{-t}\sqrt{\gamma}dtd\theta-\int u^2e^{-t}\partial_t(\ln(\sqrt{\gamma})\sqrt{\gamma}dtd\theta.
     \end{equation}
Now since  $|\partial_t\ln\sqrt{\gamma}|\leq Ce^t$ for $|T_0|$ large enough we can deduce :
\begin{equation}\label{d2}
 \int\partial_t(u^2)e^{-t}\sqrt{\gamma}dtd\theta \geq c  \int u^2e^{-t}\sqrt{\gamma}dtd\theta.
\end{equation}
Combining \eqref{d1} and  \eqref{d2} gives  
\begin{eqnarray*}\label{prout}
 c^2 \int u^2e^{-t}\sqrt{\gamma}dtd\theta&\leq& 4\int\left(\partial_tu\right)^2e^{-t}\sqrt{\gamma}dtd\theta\\
&\leq&4e^{-T_1}\int\left(\partial_t u\right)^2\sqrt{\gamma}dtd\theta.\end{eqnarray*}

\noindent Finally, droping all terms except   $\|2\tau f^\prime\partial_tu\|^2$ in \eqref{ssu}  gives :

\begin{eqnarray*}
C^{\prime}I^2\geq \tau^2 \delta^2 \|r^{-1}u\|^2.
\end{eqnarray*}
Inequality \eqref{ssu} can then be replaced by :  
\begin{equation}\begin{array}{rcl}
 I^2 &\geq &C_1\tau\int|\partial_t u|^2f^{\prime^{-3}}\sqrt{\gamma}dtd\theta+C_1\tau\int|f^{\prime\prime}|\cdot|D_\theta u|^2f^{\prime^{-3}}\sqrt{\gamma}dtd\theta\\
&+ &C_1\tau^{3}\int|f^{\prime\prime}|\cdot|u|^2f^{\prime^{-3}}\sqrt{\gamma}dtd\theta+C_1\tau^2 \delta^2\int|u|^2{f^\prime}^{-3}\sqrt{\gamma}dtd\theta.
\end{array}\end{equation}
Then  following the end of the proof of lemma \ref{lm}, we can state that for $\tau\geq C_2$ with $C_2$ great enough and depending only 
on $(M,g)$ the following estimate holds: 
\begin{equation}\label{rouge}
\begin{split}
C\left\|r^2e^{\tau\phi}\Delta u r^{-n/2}\right\|^2
&
\geq
\tau^3\left\|\sqrt{|f^{\prime\prime}(\mathrm{ln}\:r)|}e^{\tau\phi}ur^{-n/2}\right\|^2
\\
+\ \tau^2\delta\left\| r^{-\frac{1}{2}}e^{\tau\phi}ur^{-n/2}\right\|^2
&
+
\tau\left\| r\sqrt{|f^{\prime\prime}(\mathrm{ln}\:r)|}e^{\tau\phi}|\nabla u|^2r^{-n/2}\right\|^2.
\end{split}\end{equation}


\non We will now derive from \eqref{rouge}, the  Carleman estimate on $\Delta+W$. If $W$ is a bounded function, one has from triangular inequality : 

\begin{multline*}C\left\|r^2e^{\tau\phi}\left(\Delta u +Wu\right)r^{-\frac{n}{2}}\right\|^2 \geq
\frac{C}{2}\left\|r^2e^{\tau\phi}\left(\Delta u +Wu\right)r^{-\frac{n}{2}}\right\|^2\\
-C\|W\|^2_\infty\cdot\left\|r^2e^{\tau\phi}ur^{-\frac{n}{2}}\right\|^2\phantom{aaaaaaaaaa}
\end{multline*}

Since for $R_0$ small enough one has $r^2\leq \sqrt{|f^{\prime\prime}(\ln r)|}$ from \eqref{f},
 assuming additionally  that $\tau^{3}\geq C_1\|W\|^2_\infty,$ we can use \eqref{rouge} to absorb the negative term and conclude the proof. 
 
\end{proof}

\subsubsection{Special choice of weight function}
In this paragraph we derive a special case of theorem \ref{ticsl} which will be useful for the remaining of this paper. Let $\varepsilon$ be a real number such that 
$0<\varepsilon<1$.  As in previous example we consider on $]-\infty,T_0]$ the function defined by 
$$f(t)=t-e^{\varepsilon t}.$$
Therefore one can easily chech that \begin{eqnarray*}
                                     & 1-\varepsilon \leq  f^{\prime}(t) \leq 1,\\
				     &\varepsilon^{-1}|f^{(3)}(t)\leq  -f^{\prime\prime}(t)\leq\varepsilon^2, \\
				     & \lim_{t\rightarrow-\infty}f^{\prime\prime}(t)e^{-t}=+\infty,
                                    \end{eqnarray*}
and then $f$ statifies the properties  \eqref{f}. Therefore we can apply theorem \ref{ticsl} with the weight function $\phi(x)=-f(\ln r)=-\ln r+r^{\varepsilon}$. 
Now we notice that 
$$e^{\tau\phi}=e^{-\tau\ln r}e^{\tau r^\varepsilon}.$$
The point is that we can now obtain a Carleman estimate with the usual $L^2$ norm. Indeed one has $$e^{\tau\phi}r^{-\frac{n}{2}}=e^{(\tau+\frac{n}{2})\phi}e^{-\frac{n}{2}r^\varepsilon}$$ 
and therefore for $r$ small enough  
 $$\frac{1}{2}e^{(\tau+\frac{n}{2})\phi}\leq e^{\tau\phi}r^{-\frac{n}{2}}\leq e^{(\tau+\frac{n}{2})\phi}.$$ 
Then we have 
\begin{cor}
 \label{tics}
 There exist positive constants $R_0, C,C_1,C_2$, which depend only on  $M$ and $\varepsilon$, such that, 
for any $\:W\in L^\infty(M)$, $x_0\in M$, any \\ $u\in \Cc^\infty_0\left(B_{R_0}(x_0)\setminus B_{\delta}(x_0)\right)$ and 
 $\tau \geq C_1\|W\|_{\infty}^{\frac{2}{3}}+C_2$, one has  
\begin{equation}\label{Siv2}\begin{array}{rcc}
C\left\|r^2e^{\tau\phi}\left(\Delta u +Wu \right)\right\|^2&\geq& \tau^3\left\|r^{\frac{\varepsilon}{2}}e^{\tau\phi}u\right\|^2 \\
+\ \tau^2\delta\left\|r^{-\frac{1}{2}}e^{\tau\phi}u\right\|^2&+& \tau\left\|r^{1+\frac{\varepsilon}{2}}e^{\tau\phi}\nabla u\right\|^2.
\end{array}\end{equation}

\end{cor}

We emphasis that  $\varepsilon$ is fixed and its value will not have any influence on ours statements as long as $0<\varepsilon<1$. 
Therefore we may and will omit the dependency of our constants on  $\varepsilon$.

\subsection{Three balls inequality}
We now want to derive from our Carleman estimate, a control on the local behaviour of solutions. We will first give and proove an Hadamard
 three circles type theorem.  To obtain such result the basic idea is to apply Carleman estimate to $\chi u$ where 
$\chi$ is an apropriate cut off functions and $u$  a solution of \eqref{E}. 
\begin{prop}[Three balls inequality]\label{tst}
There exist positive constants $R_1$, $C_1$, $C_2$ and $0<\alpha<1$  wich depend only on $(M,g)$ such that, if $u$ is a solution to \eqref{E} with $W$ of class $L^\infty$,  
then for any $R<R_1$, and any $x_0\in M, $ one has 
\begin{equation}\label{ttc}
\| u\|_{B_R(x_0)}\leq e^ {C_1 \|W\|_{\infty}^{2/3}+C_2}\|u\|_{B_{\frac{R}{2}}(x_0)}^\alpha\| u\|_{B_{2R}(x_0)}^{1-\alpha}.
\end{equation}
\end{prop}

\begin{proof}
Let $x_0$ a point in $M$. Let $u$ be a solution 
to \eqref{E} and $R$ such that $0<R<\frac{R_0}{2}$ with $R_0$ as in corollary \ref{tics}.We will denote by $\|v\|_{R_1,R_2}$ the $L^2$ norm of $v$ on the set  
 $A_{R_1,R_2}:=\{x\in M ;\:R_1\leq r(x)\leq R_2\}$.  
Let $\psi\in\Cc^{\infty}_0(B_{2R})$, $0\leq\psi\leq1$, a function with the following properties:
\begin{itemize}
\item $\psi(x)=0$ if $r(x)<\frac{R}{4}$ or  $r(x)>\frac{5R}{3}$,
\item[$\bullet$] $\psi(x)=1$ if $\frac{R}{3}<r(x)<\frac{3R}{2}$,
\item[$\bullet$] $|\nabla\psi(x)|\leq \frac{C}{R}$,
 \item[$\bullet$] $|\nabla^{2}\psi(x) |\leq \frac{C}{R^2}$.
\end{itemize}
First since the function $\psi u$ is supported in the annulus $A_{\frac{R}{3},\frac{5R}{3}}$, we can apply estimate \eqref{Siv2} of theorem \ref{tics}. 
In particular we have, since the quotient between $\frac{R}{3}$ and $\frac{5R}{3}$ don't depend on $R$.  :
\begin{equation}\label{debut}C\left\|r^2e^{\tau\phi}\left(\Delta \psi u+2\nabla u\cdot\nabla\psi\right)\right\|\geq\tau\left\|e^{\tau\phi}\psi u\right\|.\end{equation}

\non Assume that $\tau \geq 1$, and use properties of $\psi$ to get  : 
\begin{equation}\label{3c1}\begin{array}{rcl}
\|e^{\tau\phi}u\|_{\frac{R}{3},\frac{3R}{2}} &\leq  &C\left(\|e^{\tau\phi}u\|_{\frac{R}{4},\frac{R}{3}}+\|e^{\tau\phi}u\|_{\frac{3R}{2},\frac{5R}{3}}\right) \\
&+&C\left(R\|e^{\tau\phi}\nabla u\|_{\frac{R}{4},\frac{R}{3}}+R\|e^{\tau\phi}\nabla u\|_{\frac{3R}{2},\frac{5R}{3}}\right).\end{array}\end{equation}

\non Recall that  $\phi(x)=-\ln r(x)+r(x)^\varepsilon$. In particular $\phi$ is radial and decreasing (for small $r$). Then one has, 
$$\begin{array}{rcl}
\|e^{\tau\phi}u\|_{\frac{R}{3},\frac{3R}{2}}&\leq&
 C\left(e^{\tau\phi(\frac{R}{4})}\|u\|_{\frac{R}{4},\frac{R}{3}}+e^{\tau\phi(\frac{3R}{2})}\|u\|_{\frac{3R}{2},\frac{5R}{3}}\right)\\&+&C\left(Re^{\tau\phi(\frac{R}{4})}\|\nabla u\|_{\frac{R}{4},\frac{R}{3}}+Re^{\tau\phi(\frac{3R}{2})}\|\nabla u\|_{\frac{3R}{2},\frac{5R}{3}}\right).
\end{array}$$
Now we recall the following elliptic estimates : since $u$ satisfies \eqref{E} then it is not hard to see that :
 \begin{equation}\label{lm1}\|\nabla u\|_{(1-a)R}\leq C\left(\frac{1}{(1-a)R}+\|W\|^{1/2}_{\infty}\right)\|u\|_{R}, \:\:\ \mathrm{for}\  \:0<a<1.   \end{equation}
Moreover since $A_{R_1,R_2}\subset B_{R_2}$, using formula \eqref{lm1} and properties of $\phi$ gives 

$$e^{\tau\phi(\frac{3R}{2})}\| u\|_{\frac{3R}{2},\frac{5R}{3}}\leq C\left(\frac{1}{R}+\|W\|^{1/2}_{\infty}\right)e^{\tau\phi(\frac{3R}{2})}\|u\|_{2R}.$$
Using  (\ref{3c1}) one has :
 \begin{equation*}\label{ind}
\|u\|_{\frac{R}{3},R} \leq C(\|W\|^{1/2}_{\infty}+1)\left( e^{\tau(\phi(\frac{R}{4})-\phi(R))}\|u\|_{\frac{R}{2}}+e^{\tau(\phi(\frac{3R}{2})-\phi(R))}\|u\|_{2R}\right).
  \end{equation*}
  Let $A_R=\phi(\frac{R}{4})-\phi(R)$ and  $B_R=-(\phi(\frac{3R}{2})-\phi(R))$.
  From the properties of  $\phi$, we have $0<A^{-1}\leq A_R\leq A$ and $0<B\leq B_R\leq B^{-1}$ where $A$ and $B$ don't depend on $R$. 
We may assume that $C(\|W\|^{1/2}_{\infty}+1)\geq 2$. Then we can add $\|u\|_{\frac{R}{3}}$ to each side and bound it in the right hand side by  
$C(\|W\|^{1/2}_{\infty}+1)e^{\tau A}\|u\|_{\frac{R}{2}}$. We get :
  
\begin{equation}\label{3c2}
\|u\|_{R} \leq  C(\|W\|^{1/2}_{\infty}+1)\left(e^{\tau A}\|u\|_{\frac{R}{2}}+e^{-\tau B}\|u\|_{2R}\right).
  \end{equation} 
Now we want to find $\tau$ such that $$C(\|W\|^{1/2}_{\infty}+1)e^{-\tau B}\|u\|_{2R}\leq \frac{1}{2}\|u\|_{R}$$
wich is true for $\tau \geq -\frac{1}{B}\ln\left(\frac{1}{2C(\|W\|^{1/2}_{\infty}+1)}\frac{\|u\|_R}{\|u\|_{2R}}\right).$ 
Since $\tau$ must also satisfy  $$\tau \geq C_1\|W\|_{\infty}^{2/3}+C_2,$$
we choose
\begin{equation} \label{tau2}
\tau = -\frac{1}{B}\ln\left(\frac{1}{2C(\|W\|^{1/2}_{\infty}+1)}\frac{\|u\|_R}{\|u\|_{2R}}\right)+C_1\|W\|_{\infty}^{2/3}+C_2.
\end{equation}

\noindent Up to a change of constant we may assume that $C(\|W\|^{1/2}_{\infty}+1)\leq C_1\|W\|_{\infty}^{2/3}+C_2$ then we can deduce from \eqref{3c2} that :

\begin{equation}\label{last}
\|u\|_R^{\frac{B+A}{B}}\leq e^{C_1\|W\|_{\infty}^{2/3}+C_2}\|u\|_{2R}^{\frac{A}{B}}\|u\|_{\frac{R}{2}},
\end{equation}
 
\noindent  Finally define  $\alpha=\frac{A}{A+B}$ and taking exponent $\frac{B}{A+B}$ of \eqref{last} gives the result. 

\end{proof}

\subsection{Doubling estimates}
Now we intend to show that the vanishing order of solutions to $\eqref{E}$ is everywhere bound by $C_1\|W\|^{\frac{2}{3}}_\infty+C_2$. This is an immediate consequence of the following :
\begin{thm}[doubling estimate]\label{dor}

There exist two positive constants  $C_1$ and $C_2$,  depending only on  $M$ such that : if $u$ is a solution to   \eqref{E} on $M$ with $W$ of class $\Cc^1$
then  for any $x_0$ in $M$ and any  $r>0$, one has
\begin{equation}\label{do}\|u\|_{B_{2r}(x_0)}\leq e^{C_1\|W\|_{\infty}^{2/3}+C_2}\|u\|_{B_r(x_0)}.
\end{equation}
\end{thm}
\vspace{0,5cm}
\begin{Rq}
Using standard elliptic theory to bound the $L^\infty$  norm of $|u|$ by a multiple of its $L^2$ norm on a greater ball
(see by example \cite{GT} Theorem 8.17 and problem 8.3), 
 show that the doubling estimate is still true with the $L^\infty$ norm :
\begin{equation}
\|u\|_{L^\infty(B_{2r}(x_0))}\leq e^{C_1\|W\|_{\infty}^{2/3}+C_2}\|u\|_{L^\infty(B_r(x_0))}.
\end{equation}

\end{Rq}

\noindent  To prove the theorem \ref{dor} we need to use the standard overlapping chains of balls argument (\cite{DF1,JL,K}) to show :  
 \begin{prop}\label{cor1}
For any $R>0$ their exists $C_R>0$ such that for any  $x_0\in M$, any $W\in \Cc^1(M)$ and any solutions $u$ to \eqref{E} :
$$\| u\|_{B_R(x_0)}\geq e^{-C_R(1+\|W\|_{\infty}^{2/3})}\| u\|_{L^2(M)} .$$
\end{prop}

\begin{proof}
We may assume without loss of generality that $R<R_0$, with $R_0$ as in the three balls inequality (proposition \ref{ttc}).
 Up to multiplication by a constant, we can assume that $\| u\|_{L^2(M)}=1$.
 We denote by $\bar{x}$ a point in $M$ such that  $\| u\|_{B_R(\bar{x})}=\sup_{x\in M}  \|u\|_{B_R(x)}$.
 This implies that one has  $\| u\|_{B_{R(\bar{x})}}\geq D_R$, where $D_R$ depend only on $M$ and $R$. One has from proposition (\ref{ttc})  at an arbitrary point $x$ of $M$ : 

\begin{equation}\label{cop}\| u\|_{B_{R/2}(x)}\geq e^{-c(1+\|W\|_{\infty}^{2/3})}\| u\|^{\frac{1}{\alpha}}_{B_R(x)}.\end{equation} 

Let $\gamma$ be a geodesic curve  beetween $x_0$ and $\bar{x}$ and define  $x_1,\cdots,x_m=\bar{x}$ such that 
 $x_i\in\gamma$ and
 $B_{\frac{R}{2}}(x_{i+1})\subset B_R(x_i),$ for any  $i$ from $0$ to $m-1$. The number $m$  depends only on $\mathrm{diam}(M)$ and  $R$.
 Then the properties of $(x_i)_{1\leq i\leq m}$ and inequality \eqref{cop} give for all $i$, $1\leq i\leq m$ :
\begin{equation}
\|u\|_{B_{R/2}(x_i)}\geq e^{-c(1+\|W\|_{\infty}^{2/3})}\|u\|^{\frac{1}{\alpha}}_{B_{R/2}(x_{i+1})}.
\end{equation}

The result follows by iteration and the fact that $\| u\|_{B_R(\bar{x})}\geq D_R$.

\end{proof}
\begin{cor}\label{cor2}
For all $R>0$, there exists a positive constant $C_R$ depending only on $M$ and $R$ such that at any point $x_0$ in $M$ one has
\begin{equation*}
\| u\|_{R,2R}\geq e^{-C_R(1+\|W\|_{\infty}^{2/3})}\| u\|_{L^2(M)}.
\end{equation*}
\end{cor}
\begin{proof} 
 Recall that $\|u\|_{R,2R}=\|u\|_{L^2(A_{R,2R})}$ with $A_{R,2R}:=\{x; R\leq d(x,x_0)\leq 2R)\}$.
Let  $R<R_0$ where $R_0$ is from proposition \ref{3c1}, note that $R_0\leq \mathrm{diam}(M)$. 
Since $M$ is geodesically complete, there exists a point $x_1$ in  $A_{R,2R}$  
 such that $B_{x_1}(\frac{R}{4})\subset A_{R,2R}$. From proposition \ref{cor1} one has 
 $$\|u\|_{B_{\frac{R}{4}}(x_1)}\geq e^{-C_R(1+\|W\|_{\infty}^{2/3})}\| u\|_{L^2(M)}$$
 wich gives the result. 
\end{proof}
\begin{proof}[Proof of theorem \ref{dor}]

We proceed as in the proof of three balls inequality (proposition \ref{3c1}) except for the fact that now we want the first ball to become arbitrary small in front of the others.
 Let  $R=\frac{R_0}{4}$ with $R_0$ as in the three balls inequality,  let  $\delta$ such that  $0<3\delta<\frac{R}{8}$,
and  define  a smooth function $\psi$, with $0\leq\psi\leq1$  as follows: \vspace{0,3cm}
 
\begin{itemize}
\item[$\bullet$] $\psi(x)=0$ if $r(x)<\delta$ or if $r(x)>R$,
\item[$\bullet$] $\psi(x)=1$ if  $r(x)\in[\frac{5\delta}{4},\frac{R}{2}]$,
\item[$\bullet$] $|\nabla\psi(x)|\leq\frac{C}{\delta}$ and  $|\nabla^2\psi(x)|\leq\frac{C}{\delta^2}$ if $r(x)\in[\delta,\frac{5\delta}{4}]$ ,
 \item[$\bullet$] $|\nabla\psi(x)|\leq C$ and $|\nabla^2\psi(x)|\leq C$ if $r(x)\in[\frac{R}{2},R]$.\vspace{0,3cm}
\end{itemize}
Keeping appropriates terms in \eqref{Siv2} applied to $\psi u$ gives :  
\begin{equation*}
\begin{array}{rcl}
\|r^{\frac{\varepsilon}{2}}e^{\tau\phi}\psi u\|+ \tau\delta^{\frac{1}{2}}\|r^{-\frac{1}{2}}e^{\tau\phi}\psi u\|
&\leq &C\left(\|r^2e^{\tau\phi}\nabla u\cdot\nabla\psi\|+\|r^2e^{\tau\phi}\Delta\psi u\|\right). \vspace{0,15cm}\\
\end{array}
\end{equation*}
Using properties of $\psi$, one has

\begin{eqnarray*}
\|r^{\frac{\varepsilon}{2}}e^{\tau\phi} u\|_{\frac{R}{8},\frac{R}{4}}+\|e^{\tau\phi}u\|_{\frac{5\delta}{4},3\delta} \!\!&\leq  &\!\!C \left(\delta \|e^{\tau\phi}\nabla u\|_{\delta,\frac{5\delta}{4}}+\|e^{\tau\phi}\nabla u\|_{\frac{R}{2},R}\right)\\
&\!\!+\!\!&C\left(\|e^{\tau\phi} u\|_{\delta,\frac{5\delta}{4}}+\|e^{\tau\phi}u\|_{\frac{R}{2},R}\right).\nonumber
\end{eqnarray*}
 \noindent Then  (\ref{lm1}) and properties of   $\phi$, we get

\begin{equation*}\begin{array}{ccl}
e^{\tau\phi(\frac{R}{4})} \|u\|_{\frac{R}{8},\frac{R}{4}}&+&e^{\tau\phi(3\delta)}\|u\|_{\frac{5\delta}{4},3\delta} \\ 
&\leq& C(1+\|W\|_{\infty}^{1/2})\left(e^{\tau\phi(\delta)}\|u\|_{\frac{3\delta}{2}}+e^{\tau\phi(\frac{R}{3})}\|u\|_{\frac{5R}{3}}\right),\end{array}
\end{equation*}

\noindent and adding $e^{\tau\phi(3\delta)}\|u\|_{\frac{5\delta}{4}}$ to each side leads to 
\begin{equation*}\begin{array}{ccl}
e^{\tau\phi(\frac{R}{4})} \|u\|_{\frac{R}{8},\frac{R}{4}}&+&e^{\tau\phi(3\delta)}\|u\|_{3\delta}\\ &\leq &
C(1+\|W\|_{\infty}^{1/2})\left(e^{\tau\phi(\delta)}\|u\|_{\frac{3\delta}{2}}+e^{\tau\phi(\frac{R}{3})}\|u\|_{\frac{5R}{3}}\right).
\end{array}\end{equation*}
Now we want to choose $\tau$ such that  
$$C(1+\|W\|_{\infty}^{1/2}) e^{\tau\phi(\frac{R}{3})}\|u\|_{\frac{5R}{3}}\leq \frac{1}{2}e^{\tau\phi(\frac{R}{4})} \|u\|_{\frac{R}{8},\frac{R}{4}}.$$
For the same reasons than before we choose 
  $$\tau=\frac{1}{\phi(\frac{R}{3})-\phi(\frac{R}{4})}\mathrm{ln}\left(\frac{1}{2C(1+\|W\|_{\infty}^{1/2})}\frac{\|u\|_{\frac{R}{8},\frac{R}{4}}}{\|u\|_{\frac{5R}{3}}}\right)+C_1(1+\|W\|_{\infty}^{2/3}).$$
 Define $D_R=\left(\phi(\frac{R}{3})-\phi(\frac{R}{4})\right)^{-1}$; like before one has $0<A^{-1}\leq D_R\leq A$.
 Droping the first term in the left hand side, one has
  $$\|u\|_{3\delta}\leq e^{C(1+\|W\|_{\infty}^{2/3})}\left(\frac{\|u\|_{\frac{R}{8},\frac{R}{4}}}{\|u\|_{\frac{5R}{3}}}\right)^{A}\|u\|_{\frac{3\delta}{2}} $$
  Finally from corollary \ref{cor2}, define  $r=\frac{3\delta}{2}$ to have : 
  $$\|u\|_{2r}\leq e^{C(1+\|W\|_{\infty}^{2/3})}\|u\|_{r}. $$
   Thus, the theorem is proved for all $r\leq\frac{R_0}{16}$.
 Using proposition \ref{cor1} we have for $r\geq \frac{R_0}{16}$ :
\begin{equation*}\begin{array}{rcl}
\|u\|_{B_{x_0}(r)}\geq\| u\|_{B_{x_0}(\frac{R_0}{16})}&\geq& e^{-C_0(1+\|W\|_{\infty}^{2/3})}\| u\|_{L^2(M)}\\
&\geq& e^{-C_1(1+\|W\|_{\infty}^{2/3})} \|u\|_{B_{x_0}(2r)}.
  \end{array}\end{equation*}
\end{proof}
Finally theorem \ref{t1} is an easy and  direct consequence of this doubling estimates.
\section{Possible vanishing order for solutions}

The aim of this is section is to proove theorem \ref{EX} which states that our exponent $\frac{2}{3}$ on $\|W\|_\infty$ in theorem \ref{t1} is sharp. More precisely, 
 we will construct a sequence $(u_k,W_k)$ verifying $\Delta u_k+W_ku_k=0$ and such that :
\begin{itemize}
  \item[$\bullet$] In at least one point of $\Sb^2$, the function $u_k$ vanishes at an order at least 
 $C_1\|W_k\|^{\frac{2}{3}}_\infty+C_2$, where $C_1$ and $C_2$ are fixed numbers wich don't depend on $k$.                                              
\item[$\bullet$] $\lim_{k\rightarrow+\infty} \|W_k\|_{\infty}=+\infty$.
 \item[$\bullet$] The potential $W$ has compact support and 
$$\mathrm{supp}(W)\subset\left(\Sb^2\setminus(\{P\}\cup \{Q\}\right)$$                                                                                          \end{itemize}

Our construction will be inspired by the previous work of Meshkov. In \cite{M}, he have shown that one can find non trivial, 
complex valued, solutions of $\eqref{E}$ in $\R^n$ ($n\geq2$), with the following exponential decay at infinity : $\displaystyle{|u(x)|\leq e^{-C|x|^{\frac{4}{3}}}}.$  
This gives an negative answer to a question of Landis, \cite{L}.
He has also shown that $\frac{4}{3}$ is the greatest exponent for wich one can hope find non trivial solutions to \eqref{E} in $\R^n$.
\newline
Our main point is to construct on $\R^2$ a solution $u$  to $\Delta u +Wu=0$  with apropriate decay at infinity with respect to the $L^\infty$ norm of $W$.  
Then, we use stereographic projection to obtain theorem \ref{EX}. This suggests in particular that we not only need the function $W$ to be bounded, 
since we have to pull it back to the sphere, 
but we need the  stronger condition $$|W(x,y)|\leq \frac{C}{1+|x|^2+|y|^2}.$$ 
It actually appears that we can obtain  $W$ with compact support. 
We recall that our construction relies crucially on the fact that $W$ and $u$ are complex valued fonctions.\newline

From now on, we will denote by $(r,\varphi)$ polar coordinates of a point in $\R^2$ and by $[a,b]$ 
the annulus $[a,b]:=\{(r,\varphi)\in\R^+\times[0,2\pi[, a\leq r\leq b\}$.
We have the following on $\R^2$:
\begin{prop}
Let  $\rho>0$ a real number and $N$ an integer. For $\rho$ and $N$ great enough their exits  $u$ and $W$  verifying $\Delta u +W u=0$ on $\R^2$
such that :
\begin{itemize}
\item[$\bullet$] $C_2|x|^{-N}\leq |u(x)|\leq C_1|x|^{-N}$
\item[$\bullet$] $\|W\|_{L^\infty}\leq CN^{\frac{3}{2}}$
\item[$\bullet$] $W$ has support in  $]\frac{\rho}{2},\rho+6]$
\end{itemize}                  
\end{prop}
\begin{proof}\phantom{e}\ \newline
Let $\rho\geq 1$, without loss of generality we suppose that $N$ is the square of an integer. Let  define $\delta=\frac{1}{\sqrt{N}}$.
  For $j$ an integer from  10 to $\sqrt{N}$, we note  $\ n_j=j^2$ and $k_j=2j+ 1$ and $\rho_j=\rho+6(j-1)\delta$. 
 On  each annulus $A_j=[\rho_{j},\rho_{j+1}]$  we construct a solution of  $\Delta u+W u=0$ such that  $|u|=a_jr^{-n_j}$
   on the sphere $\{\rho_j\}$ and
  $|u|=a_{j+1}r^{-n_j-k_j}=a_{j+1}r^{-n_{j+1}}$ on the sphere $\{\rho_{j+1}\}$.  
 We set  $a_{10}=1$, the numbers $a_j$, for $j>10$, still have to be defined. 
When working  in the annulus  $A_j$, we will write  $a,n,k,\rho$ instead of  $a_j,n_j,k_j,\rho_{j}$. 
Before proceeding to the proof it is useful to notice that 
\begin{equation}\label{prop}\begin{array}{c}
                 k\leq C_1\sqrt{n} \leq C_1 \sqrt{N} \\
		\frac{n}{k}\leq C_2\sqrt{N}\\
		\frac{1}{\delta}= \sqrt{N}\\
                 k\delta\leq 3
                \end{array}\end{equation}

To build $(u,W)$ in $A_j$, it is convenient  to divide the construction into four steps corresponding to the four annulus $[\rho,\rho+2\delta], [\rho+2\delta,\rho+3\delta], [\rho+3\delta,\rho+4\delta], [\rho+4\delta,\rho+6\delta].$ \ve\\
\emph{{\bf Step I. Construction on the annulus $[\rho,\rho+2\delta]$ }}\\

Define $$u_1=ar^{-n}e^{-in\varphi},$$ and $$u_2=-br^{-n+2k}e^{iF(\varphi)},$$ with 
 $b=a(\rho+\delta)^{-2k}$ such that $\left|\frac{u_2(\rho+\delta)}{u_1(\rho+\delta)}\right|=1$,
the function $F(\varphi)$ will be made explicit later. 
For $m$ from $0$ to $2n+2k-1$, we set $\varphi_m=\frac{2\pi m}{2n+2k}$ and $T=\frac{\pi}{n+k}$. 

 We will need the following 
\begin{lm}[\cite{M}, p. 350]\label{M1}
Their exists a real function $h$ of class $\Cc^2$, periodic of period $T$, with the following properties :\ve

\begin{itemize}
\item \begin{bulletequation}\label{h1}|h(\varphi)|\leq 5kT, \ \  \forall \varphi\in\R,\end{bulletequation}
\item \begin{bulletequation}\label{h2}|h^{'}(\varphi)|\leq 5k, \ \  \forall \varphi\in\R,\end{bulletequation}
\item \begin{bulletequation}\label{h3}|h^{''}(\varphi)|\leq Ckn, \ \  \forall \varphi\in\R,\end{bulletequation}
\item  \begin{bulletequation}\label{h4}h(\varphi)=-4k(\varphi-\varphi_m),\ \ \mbox{ for}\ \ \varphi_m-\frac{T}{5}\leq \varphi\leq \varphi_m+\frac{T}{5}.\end{bulletequation}
\end{itemize}
\end{lm}
\non Let now choose $F$ as follows  \begin{equation}\label{F}F(\varphi)=(n+2k)\varphi+h(\varphi).\end{equation} We also consider two  smooths radial functions $\psi_1$, $\psi_2$ with the following properties :\ve

\begin{itemize}
\item \begin{bulletequation}\label{psi1}0\leq\psi_i\leq1,\end{bulletequation}
\item  \begin{bulletequation}\label{psi2}  |\psi^{(p)}_i|\leq C\delta^{-p},\end{bulletequation}
\item   \begin{bulletequation}\label{psi3} \psi_1=1 \ \mbox{on} \ [\rho,\rho+\frac{5}{3}\delta]\  \mbox{ and} \  \psi_1=0\  \mbox{on}\ [\rho+(2-\frac{1}{10})\delta,\rho+2\delta],\end{bulletequation}
\item \begin{bulletequation}\label{psi4}\psi_2=0\ \mbox{on}\ [\rho,\rho+\frac{1}{10}\delta]\ \mbox{and}\  \psi_2=1\ \mbox{ on}\ [\rho+\frac{1}{3}\delta,\rho+2\delta].\end{bulletequation}
\end{itemize}

\non Finally  we set $u=\psi_1u_1+\psi_2u_2$.   
The estimate on $[\rho,\rho+2\delta]$ is divided among four cases.\\

\emph{Step I.a. The set $[\rho,\rho+\frac{\delta}{3}]$}\\

\non We have  $u=u_1+\psi_2u_2$, and :  $|\frac{u_2}{u_1}|=\left(\frac{r}{\rho+\delta}\right)^{2k}.$\ve 
\newline Then we introduce  $\alpha$ as an  upper bound of $|\frac{u_2}{u_1}|$ : 
 $$\left|\frac{u_2}{u_1}\right|\leq \left(\frac{\rho+\frac{\delta}{3}}{\rho+\delta}\right)^{2k}:=\alpha<1.$$
We notice that
  $$|u|\geq |u_1|-|\psi_2u_2|\geq|u_1|-|u_2|\geq \frac{1-\alpha}{\alpha}|u_2|,$$
Since $\alpha=e^{2k\ln(1-\frac{2\delta}{3(\rho+\delta)})}$, a computation  of $\frac{\alpha}{1-\alpha}$ leads to the following inequality:
 \begin{equation}\label{u_2u}|u_2|\leq\frac{3(\rho+\delta)}{4k\delta}|u|.\end{equation}

\noindent Now we compute $\Delta u$ :  
\begin{equation}\label{Deltau}\Delta u=\psi_2\Delta u_2+2\frac{d\psi_2}{dr}\frac{\partial u_2}{\partial r}+\left(\frac{1}{r}\frac{d\psi_2}{dr}+\frac{d^2\psi_2}{dr^2}\right)u_2.\end{equation}
First we estimate $\Delta u_2$. One has 
\begin{eqnarray*}
  \Delta u_2 &=&\frac{1}{r^2}\left((-n+2k)^2+iF^{\prime\prime}-(F^{\prime})^2 \right)u_2 .\end{eqnarray*}
From the defintion of $F$ \eqref{F} this gives : 
\begin{equation}\Delta u_2=\frac{1}{r^2}\left[-8kn+ih^{\prime\prime}-2(n+2k)h^\prime-(h^\prime)^2\right]u_2.\end{equation}

\non Using properties of $h$ , we obtain  
\begin{eqnarray}\label{Deltau2u2}
|\Delta u_2|&\leq& \frac{1}{r^2}\left( 8kn+Ckn+C_2k^2\right)|u_2|
\end{eqnarray}
Then from inequality \eqref{u_2u} we have ,
$$|\Delta u_2|\leq \frac{1}{r^2}Ckn\frac{3(\rho+\delta)}{k\delta }|u| \nonumber\\,$$
 Now \eqref{prop} and $\frac{\rho+\delta}{r^2}\leq1$ gives:
$$|\Delta u_2|\leq C\frac{n}{\delta}|u|\leq N^{3/2}|u|$$


\non Now we estimate the other terms in the right and side of \eqref{Deltau}. The conditions on $\psi$ \eqref{psi2} and estimate of $u_2$ \eqref{u_2u} lead to  
 \begin{equation}\left|\frac{\partial \psi_2}{\partial r}\frac{\partial u_2}{\partial r}\right|\leq \frac{Cn}{k\delta^2}|u|.\end{equation}
Similarly, we have 
\begin{eqnarray}
 \left|\left(\frac{d\psi_2}{dr}+\frac{d^2\psi_2}{dr^2}\right)\frac{1}{r}u_2\right| 
&\leq&C\frac{1}{\delta^3}|u|
\end{eqnarray}
Finally,
 $$|\Delta u|\leq C(\frac{n}{\delta}+\frac{n}{k\delta^2}+\frac{1}{\delta^3})|u|.$$  Then we have shown, using relations \eqref{prop} beetween
 $k$, $n$, $\delta$ and $N$ that :
\begin{equation}
 |\Delta u|\leq C N^{\frac{3}{2}}|u|.
\end{equation}\\

\emph{Step I.b.  $[\rho+\frac{5}{3}\delta,\rho+2\delta] $ }\\

We want to estimate  $\frac{|\Delta u|}{|u|}$, with  $u=u_2+\psi_1u_1$.
First since, $$\frac{|u_2|}{|u_1|}=\frac{b}{a}r^{2k}\geq\left(\frac{\rho+\frac{5}{3}\delta}{\rho+\delta}\right)^{2k}=\beta>1,$$ 
\non we have 
\begin{equation}\label{uu_2}|u|\geq |u_2|-|\psi_1u_1|\geq |u_2|-|u_1|\geq\left(\frac{\beta-1}{\beta}\right)|u_2|>0.\end{equation}
\non One also has
\begin{equation}\label{DeltauIb}\Delta u=\Delta u_2+2\frac{d\psi_1}{dr}\frac{\partial u_1}{\partial r}+\left(\frac{1}{r}\frac{d\psi_1}{dr}+\frac{d^2\psi_1}{dr^2}\right)u_1.\end{equation}
\non First since $r\geq\rho\geq1$ one has from \eqref{uu_2} 
\begin{equation}\label{Deltau_2uIb}|\Delta u_2|\leq \frac{Ckn}{r^2}|u_2|\leq Ckn\frac{\beta}{\beta-1}|u|.\end{equation}
Then we estimate $\frac{\beta}{\beta-1}$. On the real set  $[0,\frac{4}{3}]$, we have  : 
  $$\begin{array}{rcccl}
(3/4)\ln(4x/3)&\leq& \ln(1+x)&\leq& x\\
1+x&\leq& e^x&\leq& 1+3/4(e^{4/3}-1)x
  \end{array} $$

\noindent Now from the definition of $\beta$ and
since  $k\delta\leq 3$ this estimates gives,  
$$\frac{\beta}{\beta-1}\leq 
\frac{C}{k\delta},$$
then from \eqref{Deltau_2uIb}  we have :   \begin{equation}\label{Ib1}|\Delta u_2|\leq C\frac{n}{\delta}|u|.\end{equation}
Moreover we also have $$\left|\frac{\partial u_1}{\partial r}\frac{d\psi_1}{dr}\right|\leq\frac{C}{\delta}\left|\frac{\partial u_1}{\partial r}\right|,$$
with
  $$\left|\frac{\partial u_1}{\partial r}\right|\leq Cn|u_1|\leq C\frac{n}{\beta -1}|u|\leq\frac{Cn}{k\delta}|u|,$$ 
  thus
\begin{eqnarray*}
 \left|\frac{\partial u_1}{\partial r}\frac{d\psi_1}{dr}\right|&\leq& C \frac{n}{k\delta^2}|u|. \nonumber 
\end{eqnarray*}
This also gives
\begin{eqnarray} 
\left|\frac{\partial u_1}{\partial r}\frac{d\psi_1}{dr}\right|&\leq&CN^{\frac{
 3}{2}}|u|.
\end{eqnarray}
Similarly we have,
\begin{eqnarray*}
 \left|\frac{d^2\psi_1}{dr^2}+\frac{1}{r}\frac{d\psi_1}{dr}\right||u_1| &\leq &C\left(\frac{1}{\delta^2}+\frac{1}{\delta}\right)|u_1|\\
&\leq&C\frac{1}{\delta^2}|u_1|\\
&\leq&CN^{\frac{3}{2}}|u|.
\end{eqnarray*}

\non Here again we have shown $$ |\Delta u|\leq CN^{\frac{3}{2}}|u|.$$

\emph{Step I.c  $\{(r,\varphi);r\in[\rho+\frac{1}{3}\delta,\rho+\frac{5}{3}\delta], \ \varphi_m+\frac{T}{5}\leq\varphi\leq\varphi_m+\frac{4T}{5}\}$}\\

\non On this set we have $u=u_1+u_2$ and 
\begin{equation}|u|=|u_2|\left|e^{i(F(\varphi)+n\varphi)}-\frac{a}{br^{2k}}\right|.\end{equation}
\newline Let $S(\varphi)=F(\varphi)+n\varphi$,
we have $S(\varphi_m)=2\pi m$ and  $$S^\prime(\varphi)=2n+2k+h^\prime(\varphi)$$
Since $|h'(\varphi)| \leq5k$ we have 
$$S^\prime(\varphi)\geq2n-3k$$
Now since $n=j^2$ and $k=2j+1$, the condition $j\geq 10$ force $S^\prime(\varphi)>n$ for $\varphi\in [\varphi_m,\varphi_{m+1}]$.
Then for $$\varphi_m+\frac{T}{5}\leq\varphi\leq\varphi_{m+1}-\frac{T}{5}$$
 we have  $$S(\varphi)\geq S(\varphi_m+\frac{T}{5})\geq S(\varphi_m)+\frac{nT}{5}$$
Using the same argument  to get an upper bound on $S(\varphi)$ we can state that 
$$2\pi m+\frac{nT}{5}\leq S(\varphi) \leq 2\pi(m+1)-\frac{nT}{5}$$
This can be written 
$$2\pi m+\frac{n\pi}{5(n+k)}\leq S(\varphi) \leq 2\pi(m+1)-\frac{n\pi}{5(n+k)}$$
Now since $j\geq10$, we clearly have $\frac{n\pi}{5(n+k)}\geq\frac{\pi}{7}$. Then we can state that
$$\left|e^{iS(\varphi)}-\lambda\right|\geq \sin(\frac{\pi}{7}),$$ for any real number $\lambda$. This leads to \begin{equation}\label{u2u}|u|\geq \sin(\frac{\pi}{7})|u_2|.\end{equation}
Finally using that $\Delta u=\Delta u_2$, \eqref{Deltau2u2} and \eqref{u2u} we have 
\begin{eqnarray}
 |\Delta u|=|\Delta u_2|&\leq& Ckn|u_2|\leq ckn|u|\nonumber\\
|\Delta u|&\leq &CN^{\frac{3}{2}}|u|			 
\end{eqnarray}

\emph{Step I.d)  $\{r\in[\rho+\frac{1}{3}\delta,\rho+\frac{5}{3}\delta],|\varphi-\varphi_m|<\frac{T}{5}\}$}\\

Here we just need to notice that $u_2=-br^{-n+2k}e^{i(n-2k)\varphi}e^{4k\varphi_m}$ is  harmonic  and that $\psi_1=\psi_2=1$.
Then $u$ is harmonic and we simply set  $W=0$.\vspace{0,2mm}\\

\non\emph{{\bf Step II. Construction on the annulus $[\rho+2\delta,\rho+3\delta]$.}}\\

\non Recall that $u_2=-br^{-n+2k}e^{iF(\varphi)}$.
We now define  $$u_3=-br^{-n+2k}e^{i(n+2k)\varphi}.$$ To pass from $u_2$ to $u_3$,
we consider a smooth, radial function $\psi$ on $[\rho+2\delta,\rho+3\delta]$ with the following properties :
\begin{itemize}
\item \begin{bulletequation}\label{psiII1} 0\leq\psi(r)\leq 1 \end{bulletequation}
\item  \begin{bulletequation}\label{psiII2}\psi(r)=1$ if $r\leq \rho+{\frac{7}{3}\delta}\end{bulletequation}
\item \begin{bulletequation}\label{psiII3}\psi(r)=0$ if $r\geq \rho+\frac{8}{3}\delta \end{bulletequation}
\item\begin{bulletequation}\label{psiII4}|\psi^{(p)}(r)| \leq\frac{C}{\delta^p}\end{bulletequation}
\end{itemize}

\non Then we set on $[\rho+2\delta,\rho+3\delta]$, $$u=-br^{-n+2k}e^{i[\psi(r)h(\varphi)+(n+2k)\varphi]} $$
with  $h$, the fonction of lemma \ref{M1}.
\newline Now we prepare the computaiton of $\Delta u$, one has: \begin{eqnarray*} 
      \frac{\partial u}{\partial r}&=&\left(\frac{-n+2k}{r}+i\psi^{'}(r)h(\varphi)\right)u\\
\frac{\partial^2u}{\partial r^2}&=&\frac{(-n+2k)(-n+2k-1)}{r^2}u\\
&+&2i\frac{-n+2k}{r}\psi^\prime(r)h(\varphi)u+i\psi^{''}(r)h(\varphi)u-\psi^{'}(r)^2h(\varphi)^2u \\
\frac{\partial^2u}{\partial \varphi^2} &=&i\psi(r)h^{''}(\varphi)u+\left(i\psi(r)h^{'}(\varphi)+(n+2k)\right)^2u
 \end{eqnarray*}
       
\begin{eqnarray*}
 \Delta u&=&\left(\frac{(-n+2k)(-n+2k-1)}{r^2}+2i\frac{-n+2k}{r}\psi^\prime(r)h(\varphi)\right)u\\
 &+&\left( \frac{-n+2k}{r^2}+\frac{i\psi^{'}(r)h(\varphi)}{r}+i\psi^{''}(r)h(\varphi)-\psi^{'}(r)^2h(\varphi)^2\right)u\\
&+&\left(\frac{i\psi(r)h^{''}(\varphi)}{r^2}+ \frac{\left(i\psi(r)h^{'}(\varphi)+(n+2k)\right)^2}{r^2}\right)u
\end{eqnarray*}
We get 
\begin{eqnarray*}
 \Delta u &= &\left[\left(\frac{-n+2k}{r}+i\psi^{'}(r)h(\varphi)\right)^2+ih(r)\left(\frac{\psi^{'}(r)}{r}+\psi^{''}(r)\right)\right]u\\
&+&\left[\frac{i\psi(r)h^{''}(\varphi)}{r^2}-\frac{\left(n+2k+\psi(r)h^{'}(\phi)\right)^2}{r^2}\right]u
\end{eqnarray*}
After simplification, the main point is that there is no term of order $n^2$ left, we obtain : 
\begin{eqnarray*}
 |\Delta u|\leq C (kn+k^2+|\psi^{'}h|^2+n|\psi^{'}h|++|h\psi^{''}|+|\psi h^{''}|+n|\psi||h^{'}|+|\psi^2||h{'}|^2)u
\end{eqnarray*}
Now using \eqref{prop}, the  properties of $\psi$ (\ref{psiII1}-\ref{psiII4}) and of $h$ (\ref{h1}-\ref{h4}), we have
$$|\Delta u| \leq  C N^{\frac{3}{2}}|u|.$$

\non \emph{{\bf Step III. Construction on the annulus $[\rho+3\delta, \rho+4\delta]$.}}\phantom{e}\vspace{3mm}\\
\noindent Recall that  $$u_3=-br^{-n+2k}e^{i(n+2k)\varphi}.$$
\noindent In this step, we want to pass from $u_3$ to  the harmonic function:  $$u_4=-b_1r^{-n-2k}e^{i(n+2k)\varphi}.$$
\newline
Let $d=(\rho+3\delta)^{4k}$,  $b_1=bd$ and  $$g(r)=dr^{-4k}=\left(\frac{\rho+3\delta}{r}\right)^{4k}.$$ Then $|g^{p}(r)|\leq C^pk^p$, futhermore  $$g(r)\geq\left(\frac{\rho+3\delta}{\rho+4\delta}\right)^{4k}\geq e^{4k\ln(1-\frac{\delta}{\rho+4\delta})}$$
We have $0\leq \frac{\delta}{\rho+4\delta} \leq \frac{1}{4}$ and  $\ln(1-x)\geq 4\ln(\frac{3}{4})x$ for $x$ in the real set $[0,\frac{1}{4}]$,
then $$g(r)\geq e^{-16k\ln(\frac{4}{3})\frac{k\delta}{\rho+4\delta}}\geq e^{-16\ln\frac{4}{3}}\geq e^{-5}$$
We know consider a smooth radial function $\psi$, with the following properties :\ve
\begin{itemize}
\item[$\bullet$] $0\leq \psi \leq 1$
\item[$\bullet$] $|\psi^{(p)}|\leq \frac{C}{\delta^p}$
\item[$\bullet$] $\psi(r)=1$ if $r\in [\rho+3\delta,\rho+(3+\frac{1}{3})\delta]$
\item[$\bullet$] $\psi(r)=0$ if $r\in [\rho+(3+\frac{2}{3})\delta]$\ve
\end{itemize}
We let  $f(r)=\psi(r)+(1-\psi(r))g(r)$ so one has  $$e^{-5}\leq f(r)\leq 1.$$
A computation of $f^{\prime}(r)$ and $f^{(2)}(r)$ leads to the following estimates:
\begin{eqnarray*}
 |f^{'}(r)|&\leq &C\left(|\psi^{'}(r)|+|\psi^{'}(r)g(r)|+|\psi g^{'}|\right) \\
|f^{'}(r)|&\leq & \frac{C}{\delta}+Ck\leq \frac{C}{\delta},
\end{eqnarray*}
and 
$$|f^{''}(r)|\leq \frac{C}{\delta^2}.$$

We finally set 
$u=u_3f$, wich implies that  $|u|>0$ and $|u_3|\leq C |u|$. To compute $\Delta u_{3}$ we observe that 
\begin{eqnarray*}
 \frac{\partial u_3}{\partial r} &=&\frac{-n+2k}{r}u_3,\\
\frac{\partial^2u_3}{\partial r^2}&=&\frac{(-n+2k)(-n+2k-1)}{r^2}u_3,\\
&\mbox{and}&\\
\frac{\partial^2 u_3}{\partial \varphi^2} &=& -(n+2k)^2u_3.
\end{eqnarray*}
 Then we have :
$$\Delta u_3 =\frac{1}{r^2}\left((-n+2k)(-n+2k-1)+(-n+2k)-(n+2k)^2\right)u_3 $$
Thus we have the following estimates
\begin{equation}\label{Deltau_3u_3}|\Delta u_3|\leq Ckn|u_3|\leq Ckn|u|\end{equation}

\non From $\displaystyle{\left|\frac{\partial u_3}{\partial r}\right|\leq Cnu_3}$,
  and
 $\displaystyle{\left|\frac{d^{p}f}{dr^{p}}\right|\leq \frac{C}{\delta^p}}$ one can deduce that:
 $$2\left|\frac{df}{dr}\frac{\partial u_3}{\partial r}\right|\leq C\frac{n}{\delta}u_3\leq C\frac{n}{\delta}u,$$
 and $$\left|\frac{d^2f}{dr^2}+\frac{1}{r}\frac{df}{dr}\right||u_3|\leq \frac{C}{\delta^2}|u_3|\leq \frac{C}{\delta^2}|u|.$$
We then have the following estimate on $\Delta u$ : 
\begin{eqnarray*}|\Delta u| &\leq& C \left( \frac{n}{\delta}+kn+\frac{1}{\delta^2}\right)|u| \\
                  |\Delta u| &\leq & CN^{\frac{3}{2}}|u|
 \end{eqnarray*}
\emph{{\bf Step IV. Construction on the annulus  $[\rho+4\delta,\rho+6\delta]$}}\\
\non Recall that  $$u_4=-b_1r^{-(n+2k)}e^{i(n+2k)\varphi}$$ and let \begin{equation}\label{u_5}u_5=cr^{-n-k}e^{-i(n+k)\varphi}.\end{equation}
 We choose $c$ such that \begin{equation}\label{c}\left|\frac{u_5(\rho+5\delta)}{u_4(\rho+5\delta)}\right|=1\end{equation}
then $$\left|\frac{u_5}{u_4}\right|=\frac{r^k}{(\rho+5\delta)^k}.$$
Like step I we consider two smooth radial functions $\psi_4$ and $\psi_5$ with the following properties \ve
\begin{itemize}
\item[$\bullet$] $\psi_4=1$ on $[\rho+4\delta,\rho+(4+\frac{5}{3})\delta]$ and $\psi_4=0$ on $[\rho+(5,9)\delta,\rho+6\delta]$
\item[$\bullet$] $\psi_5=0$ on $[\rho+4\delta,\rho+(4,1)\delta]$ and $\psi_5=1$ on $[\rho+(4+\frac{1}{3})\delta,\rho+6\delta]\ve$
\end{itemize}
We let $$u=\psi_4u_4+\psi_5u_5.$$ Now we estimate $\Delta u$. Here again it is convenient to divide this into three steps:\vspace{2mm} \\
\non\emph{Step IV.a  $ [\rho+(4+\frac{1}{3})\delta, \rho+(4+\frac{5}{3})\delta]$}\phantom{e}\vspace{2mm}\\
We just have $\psi_4=\psi_5=1$ then $\Delta u =0$ and  we set $W=0$.\vspace{2mm}\\
\non \emph{IV.b  $ [\rho+(4+\frac{5}{3})\delta, \rho+6\delta]$}\vspace{2mm}\\
We have $u=\psi_4u_4+u_5$, then $|u|\geq |u_5|-|u_4|$. 
$$\left|\frac{u_5}{u_4}\right|=\frac{r^k}{(\rho+5\delta)^k}\geq \left( \frac{\rho+(4+\frac{5}{3})\delta}{\rho+5\delta}\right)^k\geq e^{k\ln(1+\frac{2\delta}{3(\rho+5\delta)})}$$
Let $C_1$ such that $\ln(1+x)\geq C_1x$ on  $[0,\frac{2}{15}]$.\newline
Then we get 
$$\left|\frac{u_5}{u_4}\right|\geq e^{\ds{C_1\frac{2k\delta}{3(\rho+5\delta)}}}\geq 1+C_1\frac{2k\delta}{3(\rho+5\delta)},$$  
since $|u|\geq |u_5|-|u_4|$ this gives $$|u|\geq C_1\frac{2k\delta}{3(\rho+5\delta)}|u_4|$$
and then  \begin{equation}\label{u_4u}|u_4|\leq \frac{C}{k\delta}|u|.\end{equation}
\newline We also have 
\begin{eqnarray*}
\Delta u&=
        &2\frac{d\psi_1}{dr}\frac{\partial u_4 }{\partial r}+\left(\frac{\partial^2\psi_4}{\partial r^2}+\frac{1}{r}\frac{\psi_4}{\partial r}\right)u_4.
\end{eqnarray*}
Since $|\frac{\partial u_4}{\partial r}|\leq Cn |u_4|$, we get 
\begin{eqnarray*}
 |\Delta u| &\leq & C\frac{n}{\delta}|u_4|+\frac{C}{\delta^2}|u_4|+\frac{C}{\delta}|u_4|.
%
\end{eqnarray*}
Here again we have obtain from \eqref{u_4u} 
$$|\Delta u| \leq  C N ^{\frac{3}{2}}|u|.$$

\emph{Step IV.c $[\rho+4\delta, \rho+(4+\frac{1}{3})\delta ] $ }\phantom{e}\vspace{2mm}\\
\non This step is similar to step IV.b and therefore the estimation is omitted.\newline

 Then, defining $a_{j+1}=c$ with $c$ from \eqref{u_5} and\eqref{c}, recall that $n_{j+1}=n_j+k_j=j^2+2j+1$, and define $u(r,\varphi)=u_j(r,\varphi)$ for $\rho_j\leq r\leq\rho_{j+1}$  we have construct on the set  
$$I_N=\bigcup_{10\leq j\leq \sqrt{N}} A_j=\bigcup_{10\leq j\leq \sqrt{N}}[\rho+6(j-1)\delta, \rho+6j\delta]=[\rho+54\delta, \rho+6[,$$
a solution of  $\Delta u+Wu=0$ with $|W|\leq C N^{\frac{3}{2}}$. 
Now we extend our construction to the whole $\R^2$ in the following way: 
On  $[\rho+6\delta,+\infty[$ we just keep the harmonic function $u$ obtained on the last annulus $A_{\sqrt{N}}$ : $u=a_Nr^{-N}e^{-iN\varphi}$.
On $[0,\rho+54\delta]$ we define $g(r)$ a smooth radial function such that \begin{itemize}
                                                                     \item $g(r)=r^{n_1}$ in $[0,\frac{\rho}{4}]$
								    \item $g(r)=r^{-n_1}$ in  $[\frac{3}{4}\rho,\rho+54\delta]$
								  \item $|g^{(p)}(r)|\leq C{\rho}^{-1}$, for all $r$ in $[0,\rho+54\delta]$
                                                                    \end{itemize}
 The constant $C$ is the last point doesn't depend on $N$ ( the gap from $n_1$ to $-n_1$ is fixed since $n_1=10$). Now define  $u=g(r)e^{-in_1\varphi}$. On the compact set 
$[0,\rho+10\delta]$ on easily get $\left|\frac{\Delta u}{u}\right|\leq C$ with $C$ independant on $N$. 
Finally we have constructed on $\R^2$  a solution of $\Delta u+Wu=0$ with  $\|W\|_{L^\infty}\leq CN^{\frac{3}{2}}$ and  $|u(x)|=|x|^{-N}$ for $|x| \geq \rho+6$. 

\non Then we consider the inverse of the stereographic projection : 
$$\begin{array}{rcl} \pi : \R^2&\rightarrow&\mathbb{S}^2\setminus \mathrm{N}\\
								(x,y) &\mapsto& \frac{1}{x^2+y^2+1}(2x,2y,x^2+y^2-1)
\end{array}$$
In the Chart  $(\mathbb{S}^2\setminus\mathrm{(0,0,1)},\pi^{-1})$, the canonical metric is written $$g_{_{\mathbb{S}^2}}=\left(\begin{array}{cc}
					\frac{4}{(x^2+y^2+1)^2}&0\\
					0&\frac{4}{(x^2+y^2+1)^2}
                                        \end{array}
                                        \right)$$
                                        
 On $\mathbb{S}^2\setminus\mathrm{N}$, we have   $$\Delta_{\mathbb{S}^2}=\frac{1}{4}(x^2+y^2+1)^2\Delta_{\mathbb{R}^2}$$ 
 Let $\bar{W}$ and $\bar{u}$  two  real valued functions defined on  $\mathbb{S}^2\setminus\mathrm{(0,0,1)}$ and  $C$ a positive constant. 
We consider $u=\bar{u}\circ\pi^{-1}$ and $W=\bar{W}\circ\pi^{-1}$,
 So we have \newline 
  $$\left\{\begin{array}{rcl}
 \Delta_{\mathbb{S}^2}\bar{u}&=&\bar{W}\bar{u}\\
 |\bar{W}(x)|&\leq &C
 \end{array}
 \right.                                 
 \Longleftrightarrow \left\{\begin{array}{rcl} \Delta_{\R^2}u&=&Wu\\
 			|W(x,y)|&\leq &\frac{C}{(1+x^2+y^2)^2}
 \end{array}   \right. $$          
 Since the function $W$ we have constructed is compactly supported with \newline 
 $|W(x)|\leq CN^{\frac{3}{2}}$, their exists  $C^{'}$ such that $$\frac{4}{(1+x^2+y^2)}|W(x,y)|\leq C^{'}N^{\frac{3}{2}},\:\:\forall (x,y)\in\R^2. $$
It follows that the function $\bar{u}$ is a solution of $\Delta_{\Sb^2}\bar{u}=\bar{W}\bar{u}$ on $\Sb^2$, with $\bar{u}$ vanishing
 at order $N$ at the north pole and with $N\geq\|\bar{W}\|_{\infty}^{2/3}$.
 \end{proof}

\bibliographystyle{siam}
\bibliography{quantitativeuniqueness}

\end{document}